\documentclass[12pt]{article} 

\usepackage{amsmath}
\usepackage{amsthm}
\usepackage{amsfonts}
\usepackage{mathrsfs}
\usepackage{stmaryrd}
\usepackage{setspace}
\usepackage{fullpage}
\usepackage{amssymb}
\usepackage{breqn}
\usepackage{enumitem}
\usepackage{bbold} 
\usepackage{authblk}
\usepackage{comment}
\usepackage{hyperref}
\usepackage{pgf,tikz}
\usepackage{graphicx}
\usetikzlibrary{decorations.pathreplacing,arrows}
\tikzstyle{vertex}=[circle,draw=black,fill=black,inner sep=0,minimum size=3pt,text=white,font=\footnotesize]

\newtheorem{thm}{Theorem}[section]

\newtheorem{lemma}[thm]{Lemma}

\newtheorem{proposition}[thm]{Proposition}

\newtheorem{clm}[thm]{Claim}

\newtheorem*{lemma*}{Lemma}
\newtheorem*{proposition*}{Proposition}
\newtheorem*{theorem*}{Theorem}

\newcommand\ex{\ensuremath{\mathrm{ex}}}

\newcommand\cH{{\mathcal H}}

\newcommand\cN{{\mathcal N}}
\newcommand\cP{{\mathcal P}}

\newcommand\cQ{{\mathcal Q}}

\newcommand\cS{{\mathcal S}}
\newcommand\cT{{\mathcal T}}

\def\marrow{{\boldmath {\marginpar[\hfill$\rightarrow \rightarrow$]{$\leftarrow \leftarrow$}}}}
\def\gd#1{{\sc GDANI: }{\marrow\sf #1}}

\def\pb#1{{\sc BALAZS: }{\marrow\sf #1}}

\newcommand{\ignore}[1]{}

\title{Generalized Tur\'an results for disjoint cliques}
\author{Dániel Gerbner\\ \small Alfr\'ed R\'enyi Institute of Mathematics\\
\small \texttt{gerbner.daniel@renyi.hu}}
\date{}

\begin{document}

\maketitle

\begin{abstract} The generalized Tur\'an number $\ex(n,H,F)$ is the largest number of copies of $H$ in $n$-vertex $F$-free graphs. We denote by $tF$ the vertex-disjoint union of $t$ copies of $F$.
Gerbner, Methuku and Vizer in 2019 determined the order of magnitude of $\ex(n,K_s,tK_r)$. We extend this result in three directions. First, we determine $\ex(n,K_s,tK_r)$ exactly for sufficiently large $n$. Second, we determine the asymptotics of the analogous number for $p$-uniform hypergraphs. Third, we determine the order of magnitude of $\ex(n,H,tK_r)$ for every graph $H$, and also of the analogous number for $p$-uniform hypergraphs.
\end{abstract}

\section{Introduction}

Given a graph $F$, $\ex(n,F)$ denotes the largest number of edges in $F$-free $n$-vertex graphs. The study of this function was initiated by Tur\'an \cite{T}, who showed that $\ex(n,K_{r})=|E(T(n,r-1))|$, where the Tur\'an graph $T(n,k)$ is the complete $k$-partite graph with each part of order $\lfloor n/k\rfloor$ or $\lceil n/k\rceil$. The Erd\H os-Stone-Simonovits Theorem \cite{ES1966,ES1946} states that if $F$ has chromatic number $r>2$, then $\ex(n,F)=(1+o(1))|E(T(n,r-1))|$. 
Moon \cite{moon} considered forbidding multiple vertex-disjoint copies of a clique. Given graphs $G,G'$ with disjoint vertex sets, let $G+G'$ denote the graph with vertex set $V(G)\cup V(G')$ and edge set $E(G)\cup E(G')\cup \{uv: u\in V(G), v\in V(G')\}$. Moon \cite{moon} showed that $\ex(n,tK_{r})=|E(K_{t-1}+T(n-t+1,r-1))|$.

A natural generalization of the above problems is counting other subgraphs instead of edges. Let $\cN(H,G)$ denote the number of copies of $H$ in $G$. Let $\ex(n,H,F)$ denote the largest $\cN(H,G)$ in $F$-free $n$-vertex graphs. The first such result is due to
Zykov \cite{zykov}, who showed that $\ex(n,K_s,K_r)=\cN(K_s,T(n,r-1))$. After several sporadic results on $\ex(n,H,F)$, the systematic study of these so-called generalized Tur\'an problems was initiated by
Alon and Shihelman \cite{AS}. One of their observation is that while $\ex(n,F)$ and $\ex(n,2F)$ are close to each other, the number of copies of other graphs can be far from each other if we forbid $F$ or $2F$. For example $\ex(n,K_3,C_5)=\Theta(n^{3/2})$, while $\ex(n,K_3,2C_5)=\Theta(n^2)$. This phenomenon was further examined by Gerbner, Methuku and Vizer \cite{gmv}. One of their main theorems is that $\ex(n,K_s,tK_r)=\Theta(n^{\lceil \frac{tr-s}{t-1}\rceil-1})$. We will improve this result.

Wang \cite{wang} determined $\ex(n,K_s,tK_2)$ for every $n,s,t$. Liu and Wang \cite{LW2} determined $\ex(n,K_r,2K_r)$ for sufficiently large $n$, while Gerbner and Patk\'os \cite{gp} determined $\ex(n,K_s,2K_r)$ for sufficiently large $n$. Yuan and Yang \cite{yy} obtained a threshold on $n$, and determined $\ex(n,K_3,2K_3)$ for every $n$. A result of Gerbner \cite{ger3} implies an exact result on $\ex(n,K_s,tK_r)$ in the case $s<r$ for sufficiently large. We do not state these results, as our results generalize each of them, in the case $n$ is sufficiently large.

Let us turn to hypergraphs now. We denote by $K_r^p$ the $p$-uniform complete hypergraph on $r$ vertices. Analogously to the graph case, given $p$-graphs $H$ and $F$, we denote by $\ex(n,F)$ the largest number of hyperedges in $F$-free $n$-vertex $p$-graphs, and by $\ex(n,H,F)$ the largest number of copies of $H$ in $F$-free $n$-vertex $p$-graphs. Even the asymptotics of $\ex(n,K_r^p)$ is unknown for every $2<p<r$. The exact value of $\ex(n,tK_p^p)$ is known for sufficiently large $n$. The optimal threshold on $n$ is stated in a conjecture of Erd\H os \cite{emc} and attracted a lot of researchers, see e.g. \cite{kk} and the references in it.

Liu and Wang \cite{LW} determined $\ex_r(n,K_s^r,tK_r^r)$ for every values of $s$, $r$, $t$ and sufficiently large $n$. We will use the following special cases of their results.

 \begin{lemma}[Liu and Wang \cite{LW}]\label{hgt} Let $s> t(r-1)$. Then 
we have that $\ex_r(n,K_s^r,tK_r^r)=\binom{tr-1}{s}$ for $n\ge tr-1$. Moreover, if $\cH$ is an $r$-uniform $tK_r^r$-free hypergraph with $\binom{tr-1}{s}$ copies of $K_s^r$, then $\cH=K_{tr-1}^r$.
\end{lemma}


Now we are ready to state our results.

\begin{thm}\label{exac}
Let $tr>s\ge r\ge 2$ and $t\ge 1$ be arbitrary integers and let $x=\lceil \frac{tr-s}{t-1}\rceil-1$. Then for sufficiently large $n$ we have $\ex(n,K_s,tK_r)=\cN(K_s,K_{t(r-x)-1}+T(n-t(r-x)+1,x))$.
\end{thm}

Let $x=\lceil \frac{tr-s}{t-1}\rceil-1$ and let $\cT$  be a $K_{x+1}^p$-free $p$-graph on $n-t(r-x)+1$ vertices with the most copies of $K_x^p$. For two $p$-uniform hypergraphs $\cH$ and $\cH'$ with disjoint vertex sets, we denote by $\cH+\cH'$ the $p$-uniform hypergraph containing $\cH$ and $\cH'$, that contains as additional hyperedges all the $p$-sets intersecting both $\cH$ and $\cH'$. 

\begin{thm}\label{asym}
Let $s\ge r\ge 2$ and $t\ge 1$ be arbitrary integers and let $x=\lceil \frac{tr-s}{t-1}\rceil-1$. Then we have $\ex(n,K_s^p,tK_r^p)=(1+o(1))\cN(K_s^p,K_{t(r-x)-1}^p+\cT)$. In particular, $\ex(n,K_s^p,tK_r^p)=\Theta(n^x)$.
\end{thm}
Given a $p$-graph $G$ and a subset $U\subset V(G)$, a \textit{partial $(m,U)$-blowup} of $G$ is obtained by replacing each vertex $u\in U$ with $m$ vertices $u_1,\dots,u_m$, and for each hyperedge $H$ of $G$, if it shares vertices $u^1,\dots, u^q$ with $U$, then we replace $H$ with $m^q$ hyperedges that contain for each $i\le q$ a vertex $u^{i}_j$ for some $j\le m$ and each vertex in $H\setminus U$. Let $U(m)$ denote the $|U|^m$ new vertices.

Let us consider a largest set $U\subset V(G)$ such that no partial $(m,U)$-blowup of $G$ contains $tK_r^p$, and let $b(G)=b(G,t,r)$ denote the order of $U$. Note that it is enough to check $m=t$ to determine whether the partial blowup contains $tK_r^p$. 

Observe that if $|U'|>b(H)$, then every $tK_r^p$ in the partial blowup contains at most $t$ vertex-disjoint cliques from $U'(m)$. Also, every copy of $H[U']$ in $U'(m)$ has $t$ vertex-disjoint cliques that are extended to $K_r^p$ using $H\setminus U'$ vertex-disjointly. That implies that in $H$, there are $t$ vertex-disjoint cliques in $H\setminus U'$ that extend to $K_r^p$ using some subgraph of $H[U']$.

\begin{thm}\label{magni}
    For every $p$-graph $H$, we have $\ex(n,H,tK_r^p)=\Theta(n^{b(H)})$. 
\end{thm}

We also prove a simple statement on the structure of the extremal graphs in the non-degenerate case.

\begin{proposition}\label{struc}
Assume that $K_r$ is not contained in $H$ and $t>1$. If $n$ is sufficiently large, then there exists an $n$-vertex $tK_r$-free graph containing $\ex(n,H,tK_r)$ copies of $H$ that contains $t-1$ vertices of degree $n-1$.
\end{proposition}

By a result of Alon and Shikhelman \cite{AS} the assumptions of the above proposition imply that $\ex(n,H,tK_r)=\Theta(n^{|V(H)|})$. Note that this is the trivial direction: any blow-up of $H$ is $K_r$-free and contains $\Theta(n^{|V(H)|})$ copies of $H$.

Let $H_1,\dots,H_k$ be the graphs obtained from $H$ by deleting at most $t-1$ vertices. Let $\alpha_i$ denote the number of ways $H_i$ can be extended to $H$ using $t-1$ additional vertices. More precisely, let $H'_i$ denote the graph we obtain by adding $t-1$ vertices $v_1,\dots,v_t$ to $H_i$ and connecting them to each other and to each vertex of $H_i$. Let $\alpha_i$ denote the number of copies of $H$ in $H'_i$ that contain all the vertices and edges in $V(H'_i)\setminus \{v_1,\dots,v_{t-1}\}$, i.e., that contain the original $H_i$.


Let $x(G)=\sum_{i=1}^k\alpha_i\cN(H_i,G)$. If the assumptions of the above proposition hold, then $\ex(n,H,tK_r)=\max\{x(G): \text{ $G$ is a $K_r$-free $(n-t+1)$-vertex graph}\}$. In the case $H$ is complete multipartite, each $H_i$ is also complete multipartite. In that case, a theorem of Schelp and Thomason \cite{scth} shows that $x(G)$ is maximized by a complete multipartite graph. Therefore, if $H$ is complete $p$-partite with $p<r$, and $n$ is large enough, then $\ex(n,H,tK_r)=\cN(H,G)$ for a complete multipartite graph with at least $t-1$ parts of order exactly 1 (it is not hard to see that $G$ has exactly $t-1$ parts of order 1).

Another case when we can obtain an exact result is when the same $n$-vertex $F$-free graph maximizes $\cN(H_i,G)$ for every $i$. For example, $\ex(n,H,2K_3)=\cN(H,K_1+T(n-1,2))$ for any bipartite $H$ with a matching covering all the vertices. Indeed, each $H_i$ is a bipartite graph with a matching covering all but at most one of the vertices. Among triangle-free graphs, the Tur\'an graph contains the most copies of such graphs by a theorem of Gy\H ori, Pach, and Simonovits \cite{gypl}.

\section{Proofs}

Our main tool is the following simple statement about forbidden matchings  in hypergraphs.

\begin{proposition}\label{coverin}
Let $r\ge 2$ and $\cH$ be an $r$-uniform family of sets that does not contain $t$ pairwise disjoint sets. Then there is a set $A$ of at most $t-1$ vertices and a set $B$ of at most $r(2t-2)$ vertices such that each member of $\cH$ contains at least one element of $A$ or at least two elements of $B$. Moreover, if for every such pair $A,B$ we have $|A|=t-1$, then there is a pair where $B=\emptyset$.
\end{proposition}

\begin{proof}
We use induction on $t$, the base cases $t=1$ and $t=2$ are trivial. First we prove the existence of $A$ and $B$. If $\cH$ does not contain $t-1$ pairwise disjoint sets, we are done by induction, hence we can assume that $H_1,\dots,H_{t-1}$ are pairwise disjoint members of $\cH$. Then every $H\in \cH$ contains a vertex of $U=H_1\cup\dots \cup H_{t-1}$. We will pick sets $H_i'\in \cH$ that contain only one element $v_i\in H_i$ from $U$, one by one in the following way. If every $H\in \cH$ contains at least 2 vertices of $U$, we are done. Otherwise, there is $H_1'\in \cH$ that shares exactly one element with $U$, say $u_1\in H_1$. In general, after picking $H_i'$, if every member of $\cH$ shares at least 2 elements with $U\cup H_1'\cup\dots\cup H_i'$, we are done, otherwise there is a set $H_{i+1}'$ that shares only one vertex $u_{i+1}$ with $U\cup H'_1\cup \dots \cup H'_i$. Note that $u_{i+1}\in U$. If $u_{i+1}\in H_j$ for some $j\le i$, then $H_1,\dots,H_{j-1},H'_j,H_{j+1},H_{t-1},H'_{i+1}$ are $t$ pairwise disjoint sets unless $u_{i+1}=u_j$. If each member of $\cH$ contains one of $u_1,\dots,u_{i}$, we are done. Otherwise, $u_{i+1}\in H_j$ with $j>i$, we can assume without loss of generality that $u_{i+1}\in H_{i+1}$. We continue this way till we find $u_{t-1}$. Then the only intersections among $H_1,\dots, H_{t-1},H_1',\dots, H_{t-1}'$ are the pairwise intersections $H_i\cap H_i'=u_i$. Every member of $\cH$ not containing any $u_i$ has to intersect both $H_i$ and $H_i'$ for some $i$, thus contains at least two vertices of $U'=H_1\cup\dots\cup H_{t-1}\cup H_1'\cup\dots\cup H_{t-1}'$. 

For the moreover part, assume indirectly that we have defined $u_{t-1}$ and that there is a member $H\in \cH$ that does not contain any $u_i$. By the above, $H$ intersects, say, $H_{t-1}$ and $H'_{t-1}$. Let $A'=\{u_1,\dots,u_{t-2}\}$. Every member of $\cH$ that does not intersect $A'$ has to intersect each of $H_{t-1}$, $H'_{t-1}$ and $H$, thus contains at least two vertices of $H_{t-1}\cup H'_{t-1}\cup H$, hence we could pick $A'$ instead of $A$, a contradiction.
\end{proof}

We will apply the above proposition to families of vertex sets of copies of $K_r^p$ in $tK_r^p$-free hypergraphs $\cH$.
Observe that if $B$ is empty, then we can connect the at most $t-1$ vertices of $A$ to every other vertex of $\cH$ without creating a $tK_r^p$. Now we can prove Proposition \ref{struc} that we restate here for convenience.


\begin{proposition*}
Assume that $K_r$ is not contained in $H$ and $t>1$. If $n$ is sufficiently large, then there exists an $n$-vertex $tK_r$-free graph containing $\ex(n,H,tK_r)$ copies of $H$ that contains $t-1$ vertices of degree $n-1$.
\end{proposition*}

\begin{proof} Let $G$ be an $n$-vertex $tK_r$-free graph containing $\ex(n,H,tK_r)$ copies of $H$. Apply Proposition \ref{coverin} for the family of $r$-cliques of $G$ to obtain a smallest set $A$ and a set $B$ such that every $K_r$ shares at least one vertex with $A$ or at least two vertices with $B$. 
If $|A|=t-1$ then $B$ is empty, thus we can connect the $t-1$ vertices of $A$ to every other vertex without creating a $tK_r$, completing the proof.

Assume that $|A|<t-1$. We erase the edges inside $B$ to obtain $G'$. Then each copy of $K_r$ in $G'$ has to intersect $A$, thus $G'$ is $(t-1)K_r$-free. Clearly $G'$ contains at least $\cN(H,G)-\binom{2tr}{2}n^{|V(H)|-2}$ copies of $H$. We will pick a vertex $v$ and connect $v$ to every other vertex. The resulting graph $G''$ is clearly $tK_r$-free, thus contains at most $\cN(H,G)$ copies of $H$. Therefore, for every $v$ we create at most $\binom{2tr}{2}n^{|V(H)|-2}$ new copies of $H$ this way.

Let $H'$ be a graph obtained by deleting a non-isolated vertex $u$ of $H$ and $u'$ be a neighbor of $u$.

Let us consider a maximal clique with vertex set $X$ in $G'$. Then $|X|\le (t-1)r-1$. Consider the copies of $H'$ in $G'$ avoiding $X$, let $m$ be their number. Each copy has a vertex $v'$ corresponding to $u'$, and $v'$ cannot be connected to every vertex of $X$. Therefore, for some $x\in X$, at least $m/|X|\ge m/tr$ copies of $H'$ avoiding $X$ have the property that the vertex corresponding to $u'$ is not connected to $x$. It means that connecting $x$ to every vertex would create a new copy of $H$ with $x$ playing the role of $u$ (it is new since the edge $v'x$ is not in $G'$. Therefore, $m/tr\le \binom{2tr}{2}n^{|V(H)|-2}$.

There are two types of copies of $H$ in $G'$. If $H'$ contains a vertex from $X$, then there are at most $(t-1)r-1\le tr$ ways to pick that vertex, and at most $n$ ways to pick the other vertices, thus there are at most $trn^{|V(H)|-1}$ such copies of $H$. If $H'$ avoids $X$, then we have $m$ ways to pick $H'$ and at most $n$ ways to pick the last vertex, thus there are at most $tr\binom{2tr}{2}n^{|V(H)|-1}$ such copies of $H$. Then $G$ contains at most $tr\cdot n^{|V(H)|-1}+\binom{2tr}{2}n^{|V(H)|-1}+\binom{2tr}{2}n^{|V(H)|-2}$ copies of $H$. On the other hand, the blow-up of $H$ contains $\Theta(n^{|V(H)|})$ copies of $H$, a contradiction if $n$ is large enough.
\end{proof}




We will use the following lemma to prove Theorem \ref{exac}.

\begin{lemma}\label{trivi}
Let $\alpha>(k-1)/k$. If $G$ is a $K_{k+1}$-free $n$-vertex graph that contains a $K_k$-free subgraph of order at least $\alpha n$, then $\cN(K_k,G)\le \cN(K_k,T(n,k))-\Theta(n^k)$. 
\end{lemma}

\begin{proof}
Let $U$ be a $K_k$-free subgraph of $G$ of the largest order. 
Let $v$ be a vertex outside $U$. The neighborhood of $v$ is $K_k$-free, thus has order at most $|U|$. Therefore, $v$ is in at most $\cN(K_{k-1},T(|U|,k-1))$ copies of $K_k$. As each copy of $K_k$ contains a vertex outside $U$, there are at most $\cN(K_{k-1},T(|U|,k-1))(n-|U|)$ copies of $K_k$ in $G$. 
We have equality if $U$ induces the Tur\'an graph $T(|U|,k-1)$, and we add the remaining vertices as one more class. By the assumption on $|U|$, this last class $A=V(G)\setminus U$ has order at most $(1-\alpha) n$. Let us move $\frac{\alpha-(k-1)/k}{2}n$ vertices from another class $B$ to $A$ and let $G'$ be the resulting graph. The number of edges between $A$ and $B$ increases by $\Theta(n^2)$. If a copy of $K_k$ in $G$ does not contain vertices from both $A$ and $B$, it is also in $G'$. The other copies of $K_k$ in $G$ contain an edge between $A$ and $B$, and there are $\Theta(n^{k-2})$ ways to pick the other vertices. The number of ways to pick the other vertices does not change and the number of edges between $A$ and $B$ increases by $\Theta(n^2)$, thus the number of copies of $K_k$ increases by $\Theta(n^k)$. Therefore, $\cN(K_k,G)\le \cN(K_k,G')-\Theta(n^k)\le \ex(n,K_k,K_{k+1})-\Theta(n^k)= \cN(K_k,T(n,k))-\Theta(n^k)$, completing the proof.
\end{proof}

We will prove Theorem \ref{exac} and Theorem \ref{asym} together. More precisely, we start with proving Theorem \ref{asym}, and afterwards we continue using the same notation to improve the bound to an exact result in the graph case. For this reason, we denote some hypergraphs with printed capital letters in the following proof. 

Now we restate Theorems \ref{exac} and \ref{asym} together. Recall that  $\cT$  is a$K_{x+1}^p$-free $p$-graph on $n-t(r-x)+1$ vertices with the most copies of $K_x^p$; in the case $p=2$ by Zykov's theorem $\cT=T(n-t(r-x)+1,x)$. 

\begin{theorem*}
Let $tr>s\ge r\ge 2$ and $t\ge 1$ be arbitrary integers and let $x=\lceil \frac{tr-s}{t-1}\rceil-1$. Then we have $\ex(n,K_s^p,tK_r^p)=(1+o(1))\cN(K_s^p,K_{t(r-x)-1}^p+\cT)$, and in the case $p=2$ for sufficiently large $n$ we have $\ex(n,K_s,tK_r)=\cN(K_s,K_{t(r-x)-1}+T(n-t(r-x)+1,x))$.
\end{theorem*}

\begin{proof}     If we take $t$ vertex-disjoint cliques in $K_{t(r-x)-1}^p+\cT$, they together have at most $t(r-x)-1$ vertices from $K_{t(r-x)-1}^p$ and at most $tx$ vertices from $\cT$, thus less than $tr$ vertices altogether, showing that $K_{t(r-x)-1}^p+\cT$ is $tK_r^p$-free.

    For the upper bound, let $n$ be large enough and $G$ be an $n$-vertex $tK_r^p$-free $p$-graph with $\ex(n,K_s^p,tK_r^p)$ copies of $K_s^p$. Let $\cH$ denote the family of $(x+1)$-element sets of vertices that extend to a $K_s^p$. In other words, we place the vertex set of a $K_{x+1}^p$ into $\cH$ only if it is contained in a $K_s^p$.

    We claim that there is a constant $c=c(s,t,r)$ such that $\cH$ contains less than $c$ pairwise disjoint sets. 
Indeed, assume that $H_1,\dots,H_c$ are pairwise disjoint members of $\cH$. Let us pick a copy of $K_s^p$ extending $H_i$ for every $i$ and denote its vertex set by $Q_i$. Let us consider an auxiliary graph $G_0$ with vertex set $H_1,\dots,H_c$ and connect $H_i$ to $H_j$ if $Q_i$ intersects $H_j$ or $Q_j$ intersects $H_i$. Then the average degree is at most $s-x+1$ in $G_0$, since $Q_i$ intersects at most $s-x+1$ sets $H_j$.
If $c$ is large enough, we can pick an independent set $T$ of order $t$ in $G_0$, by applying Tur\'an's theorem to the complement of $G_0$. Now we can greedily extend each $H_i$ from $T$ to a $K_r^p$ using the $s-x-1$ vertices from $Q_i$, such that the resulting $t$ copies of $K_r^p$ are pairwise disjoint. Indeed, altogether we use $t(r-x-1)=t(r-\lceil\frac{tr-s}{t-1}\rceil)=tr-s-(t-1)\lceil\frac{tr-s}{t-1}\rceil+s-\lceil\frac{tr-s}{t-1}\rceil\le s-\lceil\frac{tr-s}{t-1}\rceil=s-x-1$ vertices. We go through the sets in $T$ in an arbitrary order and add $r-x-1$ new vertices. Even for the last $H_i$, we have at least $r-x-1$ unused vertices in $Q_i$.

 Applying Proposition \ref{coverin}, we obtain that there is a set $A$ of at most $c-1$ vertices and a set $B$ of at most $2cr$ vertices such that every $K_{x+1}^p$ that extend to a $K_s^p$ in $G$ either contains at least one element of $A$, or two elements of $B$. If a $K_s^p$ contains less than $s-x$ elements of $A\cup B$, then there is a $K_{x+1}^p$ in it completely avoiding $A\cup B$, a contradiction. Therefore, every $K_s^p$ contains at least $s-x$ elements of $A\cup B$. Observe that there are $O(1)$ ways to pick these elements, and at most $n$ ways to pick the other elements. This shows that $\ex(n,K_s^p,tK_r^p)=\Theta(n^x)$, giving a simpler proof of the theorem of Gerbner, Methuku and Vizer \cite{gmv} mentioned earlier, and extending it to hypergraphs.

 Let $G'$ be the $p$-graph obtained by removing $A\cup B$ from $G$.
For an $(r-x)$-subset $P\subset A\cup B$, we let $\cH(P)$ denote the family of $x$-sets in $V(G)\setminus (A\cup B)$ that extend to a $K_r^p$ with $P$. Let $\cP$ be the family of $(r-x)$-subsets of $A\cup B$ with $|\cH(P)|\ge n^{x-1/2}$.

\begin{clm}
There are no $t$ pairwise disjoint sets in $\cP$.
\end{clm}

\begin{proof} Let as assume indirectly that $P_1,\dots,P_t\in \cP$ are pairwise disjoint.
By Proposition \ref{coverin}, every $\cH(P)$ with $P\in \cP$ contains $tx$ pairwise disjoint members. For $P_1\in \cP$, we pick an arbitrary member $H_1$ of $\cH(P)$. Then $H_1\cup P_1$ intersects at most $x$ of the $tx$ pairwise disjoint members of $\cH(P_2)$, thus we can pick $H_2\in \cH(P_2)$ avoiding $H_1$. We continue this way, in general for $i\le t$ we pick $H_i\in \cH(P_i)$ avoiding $H_1\cup H_2\dots\cup\dots H_{i-1}$, this is doable since each $H_j$ intersects at most $x$ sets out of the $tx$ pairwise disjoint members of $\cH(P_i)$. Then the sets $P_i\cup H_i$ are pairwise disjoint and each induces a $K_r^p$ in $G$, a contradiction.
\end{proof}

Let us return to the proof of the theorem. Let $\cQ$ be the family of vertex sets of copies of $K_s^p$ in $V(G)$ that intersect $A\cup B$ in the vertex set of a $K_{s-x}^{r-x}$ of $\cP$,
and $\cQ'$ denote the family of vertex sets of other copies of $K_s^p$ in $V(G)$. By the above Claim and Lemma \ref{hgt}, there are at most $\binom{t(r-x)-1}{s-x}$ 
copies of $K_{s-x}^{r-x}$ in $\cP$. Let $K$ be the vertex set of a copy of $K_{s-x}^{r-x}$ in $\cP$ and consider how many ways we can extend $K$ with $x$-sets from $G'$ to obtain a member of $Q$. Let $U$ be the set of vertices outside $A\cup B$ such that every $p$-set intersecting both $K$ and $U$ is a hyperedge of $G$. To extend $K$ to a member of $Q$, we have to take a $K_x^p$ in $U$.
Clearly $|U|\le n-|A\cup B|$ and there is no $K_{x+1}^p$ in $U$. Indeed, that $K_{x+1}^p$ would extend to a $K_s^p$ with any $s-x-1$ vertices from $K$, thus it is in $\cH$, hence it intersects $A\cup B$, a contradiction. Therefore, there are at most $\cN(K_x^p,\cT)$ $x$-sets that extend $K$ to a member of $Q$.

Consider now members of $\cQ'$. They each intersect $A\cup B$ in at least $s-x$ vertices. There are $O(n^{x-1})$ members where this intersection has order more than $s-x$, as there are $O(1)$ ways to pick the vertices from $A\cup B$, and at most $n$ ways to pick each remaining vertex. Those members of $\cQ'$ that intersect $A\cup B$ in an $(s-x)$-set $S$ each contain an $(r-x)$-subset $P\subset S$ not in $\cP$, thus with $|\cH(P)|<n^{x-1/2}$. Observe that each $x$-set that extends $S$ to a $K_s$ is in $\cH(P)$, thus there are less that $n^{x-1/2}$ ways to extend each $(s-x)$-subset of $A\cup B$ to a member of $\cQ'$. We obtained $|\cQ'|=O(n^{x-1/2})$.
This implies the asymptotically sharp upper bound $(1+o(1))\binom{t(r-x)-1}{s-x}\cN(K_x^p,\cT)$. 

Let us assume from now on that $p=2$.
So far we obtained the asymptotically sharp upper bound $(1+o(1))\binom{t(r-x)-1}{s-x}\cN(K_x,T(n-|A\cup B|,x))$. Observe that we lose a constant factor of the copies of $K_r$ in the case there are less than $\binom{t(r-x)-1}{s-x}$ $(s-x)$-sets in $A\cup B$ with each $r$-subset in $\cP$. Therefore, we are done with the proof unless we have equality when applying Lemma \ref{hgt} to $\cP$. 

In that case, there is a set $C\subset A\cup B$ of $t(r-x)-1$ vertices in $G$ such that every $(r-x)$-subset of $C$ is contained in at least $n^{x-1/2}$ copies of $K_r$. In particular, they form a $K_{t(r-x)-1}$. We also lose a constant factor of the copies of $K_r$ in the case any of the $(s-x)$-subsets of $C$ is extended to a $K_s$ only by $\cN(K_x,T(n,x))-\Omega(n^x)$ copies $K_x$ in $V(G)\setminus (A\cup B)$. In particular, if the common neighborhood of the $K_{s-x}$ has order $n-\Omega(n)$, then this is the case. Thus we are done, unless each vertex of $C$ is connected to $n-o(n)$ vertices, which implies that the common neighborhood $D$ of $C$ contains $n-o(n)$ vertices. Clearly there are $\Theta(n^x)$ copies of $K_x$ in $D$, otherwise $|\cQ|=o(n^x)$, which would imply that $\cN(K_s,G)=o(n^x)$.

\begin{clm}
    $D$ is $K_{x+1}$-free.
\end{clm}
\begin{proof}  
 In $D$ there are no $t-1$ copies of $K_x$ and a copy of $K_{x+1}$ that are pairwise vertex-disjoint, because we could pick $r-x$ or $r-x-1$ vertices from $C$ vertex-disjointly to extend them to $tK_r$. If there is a copy $K$ of $K_{x+1}$ in $D$, then clearly there are $O(n^{x-1})$ copies of $K_x$ in $D$ intersecting $K$. Also, the family of copies of $K_x$ inside $D$ that is disjoint from $K$ does not contain $t-1$ pairwise vertex disjoint sets, thus we can apply Proposition \ref{coverin} to show that there are $O(n^{x-1})$ copies of $K_x$ in $D$ avoiding $K$. Therefore, there are $O(n^{x-1})$ copies of $K_x$ in $D$, a contradiction.
\end{proof}

Let us return to the proof of the theorem. We claim that every member of $\cH$ intersects $C$. Assume that there is a $K_s$ with vertex set $S$ containing $x+1$ elements not in $C$. Observe that $D$ contains $(t+s)K_x$, as otherwise $D$ contains $O(n^{x-1})$ copies of $K_x$ by Lemma \ref{coverin} and we are done. Then we have $(t-1)K_x$ vertex disjoint from $S$. We also have $(t-1)(r-x)$ vertices in $C\setminus S$, but then we can add $r-x$ distinct ones to each of the $t-1$ copies of $K_x$ to obtain $(t-1)K_r$ plus a $K_s$ vertex-disjointly. This configuration contains $tK_r$, a contradiction.

Let us consider $u\not\in C\cup D$, and let $\cS$ denote the family of vertex sets of the copies of $K_s$ that contain $u$.  Observe that $u$ is connected to at least $s-x$ vertices of $C$, as otherwise any $S\in \cS$ contains at least $x+1$ vertices not in $C$, giving us a member of $\cH$ avoiding $C$, a contradiction. This implies that the neighborhood of $u$ in $D$ is $K_x$-free. Indeed, such a $K_x$ with $u$ would extend to $K_s$ with $s-x-1$ neighbors of $u$ in $C$, giving us a member of $\cH$ avoiding $C$, a contradiction.  The number of sets $S$ in $\cS$ with  $|S\cap D|<x-1$ is $o(n^{x-2})$. Consider the sets $S$ in $\cS$ with  $|S\cap D|=x-1$, and let $\cS'$ denote the family of those intersections.

If $u$ has at least $\alpha n$ neighbors in $D$ for some $\alpha>(x-1)/x$, then by Lemma \ref{trivi} the number of copies of $K_x$ in $D\cup \{u\}$ is at most $\cN(K_x,T(n,x))-\Theta(n^x)$, a contradiction. Therefore, $u$ has at most $(x-1)n/x+o(n)$ neighbors in $D$, hence there are at most $\cN(K_{x-1},T((x-1)n/x+o(n),x))=\cN(K_{x-1},T(\lfloor (x-1)n/x\rfloor,x))+o(n^{x-1})$ copies of $K_x$ that contain $u$ and a $K_{x-1}$ from $D$. The other copies of $K_x$ that contain $u$ also contain a vertex not in $D$, thus there are $o(n^{x-1})$ such copies. This means that the number of copies of $K_s$ containing $u$ is at most $(1+o(1))\binom{t(r-x)-2}{s-x}\cN(K_{x-1},T(\lfloor (x-1)n/x\rfloor,x))$, since $u$ is connected to at most $t(r-x)-2$ vertices of $C$.

Let us first delete all the vertices of $V(G)\setminus (C\cup D)$, this way we remove at most $(1+o(1))|V(G)\setminus (C\cup D)|\binom{t(r-x)-2}{s-x}\cN(K_{x-1},T(\lfloor (x-1)n/x\rfloor,x))$ copies of $K_s$. We replace the resulting graph $G''$ with $K_{t(r-x)-1}+T(|D|,x)$, this way the number of copies of $K_s$ does not decrease. Indeed, each $K_s$ inside $C\cup D$ is built by taking a set of at least $s-x$ vertices from $C$ and a clique from $D$, and $D$ is $K_{x+1}$-free. Thus we can apply Zykov's theorem to show that the number of cliques of any order in $D$ does not decrease with this change from $G''$ to $K_{t(r-x)-1}+T(|D|,x)$. As $G''$ contains all the possible edges outside $D$, the number of copies of $K_s$ does not decrease.

After that, we add $|V(G)\setminus (C\cup D)|$ vertices to $T(|D|,x)$ and edges to obtain a Tur\'an graph on $n-t(r-x)+1$ vertices. Each vertex added creates $(1+o(1))|V(G)\setminus (C\cup D)|\binom{t(r-x)-1}{s-x}\cN(K_{x-1},T(\lfloor (x-1)n/x\rfloor,x))$ new copies of $K_s$, thus the number of copies of $K_s$ is more in this Tur\'an graph than in $G$, a contradiction completing the proof.
\end{proof}





We continue with the proof of Theorem \ref{magni}. We will use rainbow matchings in the proof. For a collection of matchings $M_1,\dots,M_k$ in an $r$-uniform hypergraph, another matching $M$ is \textit{rainbow} if for its hyperedges $m_1,\dots,m_t$ there are distinct matching $M_{i_1},\dots,M_{i_t}$ with $m_j\in M_{i_j}$.
Let $F(r,t)$ denote the maximum $k$ for which there exists a collection of $k$ matchings, each of size $t$, in some $r$-uniform hypergraph, such that there is no rainbow matching of size $t$. This problem was first investigated by Alon \cite{alon}, after Aharoni and Berger \cite{AB} studied it in $r$-partite hypergraphs. For bounds on $F(r,t)$ see e.g. \cite{psz} and the citations within. We only need that there is an upper bound on $F(r,t)$ that does not depend on $n$. Let us now restate Theorem \ref{magni} for convenience.

\begin{theorem*}
        For every $p$-graph $H$, we have $\ex(n,H,tK_r^p)=\Theta(n^{b(H)})$. 
\end{theorem*}

\begin{proof}
    To prove the lower bound, observe that for any $U\subset V(H)$, the partial $(n/|V(H)|,U)$-blowup of $H$ has at most $n$ vertices and contains $\Omega(n^{|U|})$ copies of $H$.

    To prove the upper bound, we proceed similarly to the proof of Theorems \ref{exac} and \ref{asym}. Let $G$ be an $n$-vertex $tK_r^p$-free $p$-graph and $n$ be sufficiently large. Let $\cH$ denote the family of $(b(H)+1)$-element sets of vertices that are contained in a copy of $H$. 

    We claim that there is a constant $c=c(H,t,r,p)$ such that $\cH$ contains less than $c$ pairwise disjoint sets. Indeed, assume that $H_1,\dots,H_c$ are pairwise disjoint members of $\cH$. Let us pick a copy of $H$ extending $H_i$ for every $i$ and denote its vertex set by $Q_i$. Recall that for any fixed subhypergraph of $H$ on $b(H)+1$ vertices, there are $t$ vertex-disjoint cliques in $H$ that extend to $K_r^p$ with those. It implies that $Q_i[G]$ has $t$ vertex-disjoint subhypergraphs $A_i^1\dots A_i^t$ such that each $A_i^j$ extends to a $K_r^p$ with a subhypergraph of $H_i$ in $G$.
    
    Let us consider an auxiliary graph $G_0$ with vertex set $\cH$ and connect $H_i$ to $H_j$ if $Q_i$ intersects $H_j$ or $Q_j$ intersects $H_i$. Then the average degree in $G_0$ is at most $|V(H)|-b(H)+1$, since $Q_i$ intersects at most $|V(H)|-b(H)+1$ sets $H_j$. If $c$ is large enough, we can pick an independent set $T$ of order $F(r,t)+1$ in $G_0$. That means we have $H_1,\dots,H_{F(r,t)+1}$ such that $\cup_{i=1}^{F(r,t)+1}H_i$ and $\cup_{i=1}^{F(r,t)+1}Q_i$ are disjoint. For each $A_i^j$ with $H_i\in T$ we add $r-|A_i^j|$ new vertices. This way, for each $i$, the sets $A_i^j$ form an $r$-uniform matching of size $t$, thus by the definition of $F(r,t)$, we can find a rainbow matching of size $t$. This is, without loss of generality, the sets $A_1^{j_1},\dots,A_t^{j_t}$ are pairwise disjoint. Then the pairwise disjoint sets $H_1\cup A_1^{j_1},\dots,H_t\cup A_t^{j_t}$ each contain a $K_r^p$, thus there is a $tK_r^p$ in $G$, a contradiction.


    Now we can apply Proposition \ref{coverin} to show sets $A$ and $B$ of order $O(1)$ such that each member of $\cH$ contain at least one element of $A$ or at least two elements of $B$. This implies that each copy of $H$ contains at least $|V(H)|-b(H)$ vertices from $A\cup B$. Those vertices can be chosen $O(1)$ ways, the other at most $b(H)$ vertices can be chosen $O(n^{b(H)})$ ways, completing the proof.
\end{proof}

\vskip 0.3truecm

\textbf{Funding}: Research supported by the National Research, Development and Innovation Office - NKFIH under the grants SNN 129364, FK 132060, and KKP-133819.

\vskip 0.3truecm


\end{document}